\documentclass[12pt]{amsart}
\usepackage{amsmath, amscd}
\usepackage{amsfonts}
\usepackage{amssymb}
\usepackage{amsthm}
\usepackage{verbatim}
\usepackage[mathscr]{euscript}
\allowdisplaybreaks
\usepackage{tikz}
 \usepackage{pst-grad} 
 \usepackage{pst-plot} 
\usepackage{xspace}
\usepackage{fullpage}
\usepackage{graphicx}

\def\ul{\underline}

\def\ds{\displaystyle}
\newcommand{\R}{\ensuremath{\mathbb{R}}\xspace}

\newcommand{\Z}{\ensuremath{\mathbb{Z}}\xspace}

\newcommand{\Sp}{\ensuremath{\mathbb{S}}\xspace}

\newcommand{\supp}{\ensuremath{\mbox{supp}\,}\xspace}

\newtheorem{theorem}{Theorem}[section] 
\newtheorem{lemma}[theorem]{Lemma} 
\newtheorem{proposition}[theorem]{Proposition}

\theoremstyle{definition} 

\theoremstyle{remark}
\newtheorem{remark}[theorem]{Remark}
\newtheorem*{remark*}{Remark}

\begin{document}
\title{Stability of Blow Up for a 1D model of Axisymmetric 3D Euler Equation}
\author{Tam Do}
\address{\hskip-\parindent
Tam Do\\
Department of Mathematics\\
Rice University\\
Houston, TX 77005, USA}
\email{Tam.Do@rice.edu}
\author{Alexander Kiselev}
\address{\hskip-\parindent
Alexander Kiselev\\
Department of Mathematics\\
Rice University\\
Houston, TX 77005, USA}
\email{kiselev@rice.edu}
\author{Xiaoqian Xu}
\address{\hskip-\parindent
Xiaoqian Xu\\
Department of Mathematics\\
University of Wisconsin-Madison\\
Madison, WI 53706, USA
}
\email{xxu@math.wisc.edu}

\date{\today}

\begin{abstract}
The question of the global regularity vs finite time blow up in solutions of the 3D incompressible Euler equation is a major open problem of modern applied analysis.
In this paper, we study a class of one-dimensional models of the axisymmetric hyperbolic boundary blow up scenario for the 3D Euler equation proposed by Hou and Luo \cite{HouLuo1} based on extensive
numerical simulations. These models generalize the 1D Hou-Luo model suggested in \cite{HouLuo1}, for which finite time blow up has been established in \cite{sixAuthors}.
The main new aspects of this work are twofold. First, we establish finite time blow up for a model that is a closer approximation of the three dimensional case than the original Hou-Luo model,
in the sense that it contains relevant lower order terms in the Biot-Savart law that have been discarded in \cite{HouLuo1}, \cite{sixAuthors}. Secondly, we show that the blow up mechanism is quite robust,
by considering a broader family of models with the same main term as in the Hou-Luo model. Such blow up stability result may be useful in further work on understanding the 3D hyperbolic blow up scenario.

\end{abstract}

\maketitle

\section{Introduction}
The main purpose of this paper is to contribute to the analysis of a recently discovered scenario for singularity formation in solutions of 3D Euler equation.
The 3D axisymmetric Euler equation with swirl is given by
\begin{align}\label{e1}
\partial_t\left(\frac{\omega^{\theta}}{r}\right)+u^r\left(\frac{\omega^{\theta}}{r}\right)_r+u^z\left(\frac{\omega^{\theta}}{r}\right)_z=\left(\frac{(ru^{\theta})^2}{r^4}\right)_z,
\end{align}
\begin{align}\label{e2}
\partial_t (ru^{\theta})+u^r(ru^{\theta})_r+u^z(ru^{\theta})_z=0,
\end{align}
where $u^r$ and $u^z$ can be calculated via
\begin{align}\label{BS}
u^r=\frac{\psi_z}{r}, \,\,\,
u^z=-\frac{\psi_r}{r},
\end{align}
and the stream function $\psi$ satisfies the elliptic equation
\begin{align}\label{BS1}
\frac{1}{r}\frac{\partial}{\partial r}\left(\frac{1}{r}\frac{\partial \psi}{\partial r}\right)+\frac{1}{r^2}\frac{\partial^2\psi}{\partial z^2}=\omega.
\end{align}
One can write $u^r$ and $u^z$ in terms of $\omega$ by computing the Green's function of the above elliptic PDE; more details can be found on \cite{majda2002vorticity}.

The numerical simulations performed in \cite{HouLuo1} consider fluid contained in an infinite cylinder with periodic boundary conditions in $z$ and no flux condition at the boundary of the
cylinder. The initial data is given by non-zero swirl $u^\theta$, which is odd in $z,$ while the angular vorticity is originally zero. For a particular example of such initial data, very fast
growth of angular vorticity is observed at a ring of hyperbolic points defined by the boundary of the cylinder and $z=0.$ As the first step towards rigorous analysis of this scenario, a 1D
model inspired by the numerics has been proposed in \cite{HouLuo1,HouLuo2}. We will refer to this 1D model as Hou-Luo (HL) model. The HL model is given by
\begin{align}
\label{model1}
\omega_t+u\omega_x &= \theta_x \\
\label{model2}
\theta_t+u \theta_x &=0 \\
\label{m3}
u_x =H\omega,
\end{align}
where $H$ is the Hilbert transform and the space domain is taken to be periodic, $\Sp^1$ (the $\R^1$ setting can also be considered).
One should think of the $x$ coordinate as corresponding to the $z$ direction in the original equation.
Equivalently, if $\omega$ is mean zero over the period, we can write the Biot-Savart law for $u$ as
\begin{align}
\label{HL}
u(x,t) = k\ast \omega(x,t) \quad\mbox{where}\quad k(x)= \frac{1}{\pi} \log|x|.
\end{align}
In the periodic case, $\omega$ in the formula above is extended to all real line where the convolution is applied. The convergence of the integral is understood
in the appropriate principal value sense.
In \cite{sixAuthors}, finite time blow up is shown for \eqref{model1}-(\ref{m3}) for a large class of smooth initial data.

There has been other work motivated by Hou-Luo computations and relevant to understanding the hyperbolic boundary blow up scenario. Kiselev and Sverak \cite{KS} show very fast (in fact, optimal)
growth of $\nabla \omega$ in solutions of 2D Euler equation in a geometry related to the Hou-Luo scenario. Choi, Kiselev and Yao \cite{CKY} analyzed a 1D model related to the HL model, but with a simplified
Biot-Savart law inspired by \cite{KS}. They established finite time blow up for a broad class of initial data. Hou and Liu \cite{HouLiu} have described the blow up solutions in the CKY model in more detail,
and showed that these solutions possess self-similar structure.

We note that the tradition of 1D models in fluid mechanics goes back many years. One of the earliest of these models was proposed by Constantin, Lax and Majda \cite{CLM}, and later inspired other models
\cite{deGregorio}, \cite{CCF}, see also \cite{DHRX}, \cite{HR} for recent related work. 
It is fascinating that many natural questions about solutions to these models remain unanswered. We refer the reader to \cite{sixAuthors} for a survey of this subject.

\bigskip
In this paper, our first theorem is the generalization of the results of \cite{sixAuthors} to the model with the following adjusted choice of Biot-Savart law:
\begin{align}
\label{kern}
u(x,t) = k\ast \omega(x,t) \quad\mbox{where}\quad k(x)=\frac{1}{\pi}\log\frac{|x|}{\sqrt{x^2+a^2}}.
\end{align}
It has been observed already in \cite{HouLuo1,sixAuthors} that the kernel \eqref{kern} appears naturally in the reduction of the 3D Euler equation to the 1D model of hyperbolic blow up scenario.
Nevertheless, the simpler kernel \eqref{HL} has been considered as the first step.
The difference between \eqref{kern} and the original choice \eqref{HL} is smooth, so one can expect that the properties of the equations should be similar. However, the actual proof of finite time blow up
in \cite{sixAuthors} relies on fairly fine properties of the Biot-Savart kernel, so the extension to \eqref{kern} is far from immediate.
In Section~\ref{percase}, we prove finite time blow up of solutions to the system (\ref{model1}) and (\ref{model2}) with law $\eqref{kern}$.
While we will be able to follow the framework of the blow up proof developed in \cite{sixAuthors}, many new estimates will be needed.
Similarly to \cite{sixAuthors}, the proof shows finite time blow up for a rather wide class of the initial data.

\medskip
For our second main result, we prove that the solutions to $\eqref{model1}$, $\eqref{model2}$ with even more general kernels in the Biot-Savart law exhibit finite time blow up as well. We will modify \eqref{HL} by adding a smooth function which preserves the symmetries of $\eqref{model1}$, $\eqref{model2}$ (and of the initial data). The details will appear in Section~\ref{pertblowup}. To prove blow up, roughly speaking, we isolate the ``leading term" of dynamics that
leads to blow up and persists even with a more general Biot-Savart law. The proof is quite different from the first result: the proof of finite time blow up for the Biot-Savart law \eqref{kern} relies, in the
spirit of \cite{sixAuthors}, on algebraic estimates which show that certain key quantities are positive definite. On the other hand, the more general blow up stability result is proved in a perturbative fashion, utilizing a global
bound on the $L^1$ norm of vorticity. It may appear that our second result includes the first one, but it is not literally true as in the second case we have to work
with a much more restrictive class of initial data.

\medskip
One can think of our results as strengthening the case for studying the hyperbolic blow up scenario for the 3D Euler equation. By proving singularity formation for more general Biot-Savart laws, one can view the blow up of (\ref{model1})-(\ref{m3}) as a robust phenomenon not dependent on the fine structure of the model. This may help to build a base for the next step - rigorous analysis of the higher dimensional problems.

\section{Derivation of the Model Equations}\label{deriv2}

To obtain a simplified model of \eqref{e1},\eqref{e2} the first step is to consider reduction to the 2D inviscid Boussinesq equations. This system on a half plane $\R \times [0,\infty)$ is given by
\begin{align}\label{2dbous}
\omega_t+ u^x \omega_x +u^y\omega_y &=\theta_x \\ \nonumber
\theta_t+ u^x \theta_x + u^y\theta_y &= 0
\end{align}
where $u=(u^x,u^y)$ and is derived from $\omega$ by the usual 2D Euler Biot-Savart law $u=\nabla^\perp (-\Delta_D)^{-1}\omega,$ with $\nabla^\perp =(\partial_2, -\partial_1)$ and $\Delta_D$ Dirichlet Lapalcian.
The system is classical and describes motion of 2D ideal buoyant fluid in the field of gravity. The global regularity of solutions to 2D inviscid Boussinesq system is also open. This problem is featured in the Yudovich's list of ``eleven
great problems of mathematical hydrodynamics" \cite{Yud}.

The fact that 2D inviscid Boussinesq equation is a close proxy for 3D axisymmetric Euler equation, at least away from the axis $r=0,$ is well known (see e.g. \cite{majda2002vorticity}).
Indeed, if in \eqref{e1}, \eqref{e2}, \eqref{BS}, \eqref{BS1} we re-label $\omega^\theta/r \equiv \omega,$ $r u^\theta \equiv \theta,$ $r=y,$ $z=x,$ and set $r=1$ in the coefficients, we obtain \eqref{2dbous}.
Since in the Hou-Luo scenario, the fastest growth of vorticity is observed at the boundary of the cylinder $r=1,$ and in particular away from the axis, the analogy should apply.
In \cite{sixAuthors}, to derive the HL model, the authors consider the system \eqref{2dbous} in the half-plane and restrict the system to the boundary $\{(x,y): y=0\}$ so we have $u^y=0$.
To derive a closed form Biot-Savart law for the 1D system, $\omega$ is assumed to be constant in $y$ in a boundary layer close to the boundary of width $a>0$, and zero elsewhere.
 Such assumption leads to a law defined by convolution with the following kernel:

\begin{equation*}
k(x_1) = \int_0^a \left.\frac{\partial}{\partial x_2} \right|_{x_2=0} G_D((x_1,x_2), (0,y_2))\, dy_2
\end{equation*}
where $G_D$ is the Green's function of Laplacian in the upper half-plane with Dirichlet boundary conditions. We know that
$$
G_D(z,w)=\frac{1}{2\pi}\log|z-w|-\frac{1}{2\pi}\log|z-w^*|,\quad w^*=(w_1,-w_2),
$$
and by a simple calculation one gets

\begin{equation}
u(x)=\tilde{k}\ast \omega (x),
\end{equation}
where
\begin{equation}\label{kernel}
\tilde{k}(x)=\frac{1}{\pi}\log\frac{|x|}{\sqrt{x^2+a^2}}.
\end{equation}
In \cite{sixAuthors}, the authors discard the smooth part of $\tilde{k}$ (namely, $\frac{1}{\pi}\log(\sqrt{x^2+a^2})$). In this paper we will consider $\tilde{k}$ directly or even more general perturbed kernels.

While the boundary layer assumption is strong and clearly does not hold for the higher dimensional case precisely, it is noted in \cite{HouLuo1} that the numerical simulations of the full 3D Euler equation and of the
reduced 1D model exhibit striking similarity.
Based on the numerical results about potential singularity profile for 3D axisymmetric Euler equation (\cite{HouLuo1}), we are particularly interested in the case when $\omega$ is periodic in $x$ (formerly $z$) variable
and will treat this case in the next section. The periodic assumption is not crucial; in the appendix we will outline the arguments which adjust the proof to the real line case.

We complete this section by stating a local well-posedness and a conditional regularity result that we will need later.
\medskip
The system $\eqref{model1}$, $\eqref{model2}$, $\eqref{kern}$ is locally well posed and possesses a Beale-Kato-Majda
type criterion. We formalize this below.

\medskip
\begin{proposition}
\label{lwp}
 (Local Existence and Blow Up Criteria)
Suppose $(\omega_0, \theta_0)\in H^m(\Sp^1)\times H^{m+1}(\Sp^1)$ 
where $m\ge 2$. Then there exists $T=T(\omega_0,\theta_0)>0$ such that there exists a unique classical solution $(\omega,\theta)$ of $\eqref{model1}$, $\eqref{model2}$, $\eqref{kern}$ and $(\omega,\theta)\in C([0,T]; H^m\times H^{m+1})$. In particular, if $T^*$ is the maximal time of existence of such solution then
\begin{align}
\label{BKM}
\lim_{t\nearrow T^*} \int_0^t \|u_x(\cdot, \tau)\|_{L^\infty}\, d\tau=\infty.
\end{align}
\end{proposition}

\medskip
The proof of the proposition is relatively standard. A short discussion of this topic can be found in \cite{sixAuthors}. A similar statement is also proved in detail in \cite{CKY}.
An analogous result will apply to the systems with more general Biot-Savart law that we will introduce later.

\section{The Modified Hou-Luo Kernel: Periodic Case}\label{percase}

\medskip
In this section, we prove finite time blow up of the system with the kernel given by $\eqref{kern}$ and periodic initial data. From now on,
we will refer to the kernel given by $\eqref{HL}$ as the Hou-Lou kernel, and to the kernel \eqref{kern} as the modified Hou-Luo kernel. We will denote the
velocity corresponding to the Hou-Luo kernel as $u_{HL}.$
In addition, we will consider solutions with mean zero vorticity. A straightforward calculation shows that the mean zero
property is conserved for all times for regular solutions.

Let us start by deriving a simpler expression for the Biot-Savart law in the case when the solution is periodic with period $L$. Our computations will be formal, ignoring the lack of absolute convergence of the
integrals involved; they can be made fully rigorous using standard regularization and approximation procedures at infinity.
We periodize the kernel associated with our velocity
\begin{align*}
u(x,t) &= \frac{1}{\pi} \int_{-\infty}^\infty \omega(y) \log\frac{|x-y|}{\sqrt{(x-y)^2+a^2}}\, dy = \frac{1}{\pi} \sum_{n\in \Z} \int_0^L \omega(y) \log\frac{|x-y+nL|}{\sqrt{(x-y+nL)^2+a^2}}\, dy \\
&=\frac{1}{\pi} \sum_{n\in \Z} \int_0^L \omega(y)\log|x-y+nL|\, dy\\ &\,  - \frac{1}{2\pi}\sum_{n\in \Z} \int_0^L \omega(y)\left(\log((x+ia-y)+nL)+\log((x-ia-y)+nL)\right)\, dy \\
&= \frac{1}{\pi} \int_0^L \omega(y) \log\left| (x-y)\prod_{n=1}^\infty \left(1- \frac{(\mu (x-y))^2}{\pi^2 n^2}\right) \right|\, dy  \\
&\quad -\frac{1}{2\pi} \int_0^L \omega(y)  \log\left| (x+ia-y)\prod_{n=1}^\infty \left(1- \frac{(\mu (x+ia-y))^2}{\pi^2 n^2}\right) \right| \, dy \\
&\quad -\frac{1}{2\pi} \int_0^L \omega(y)  \log\left| (x-ia-y)\prod_{n=1}^\infty \left(1- \frac{(\mu (x-ia-y))^2}{\pi^2 n^2}\right) \right| \, dy \\
&= \frac{1}{\pi} \int_0^L \omega(y)\log|\sin[\mu(x-y)]|\, dy -\frac{1}{2\pi} \int_0^L \omega(y) \log|\sin(\mu(x-ia-y))\sin(\mu(x+ia-y))|\, dy
\end{align*}
where we set $\mu=\pi/L$. In the last step we used the fact that
\[ f(z) = z \prod_{n=1}^\infty \left( 1- \left( \frac{\mu z}{\pi n} \right)^2 \right) \]
is an entire function, its zeroes coincide with those of $\sin (\mu z),$ and $\left. f'(z) \right|_{z=0} = 1.$
A quick computation leads to
\begin{align*}
\sin\mu(x-ia)\sin\mu(x+ia) &= \frac{e^{i\mu(x-ia)}-e^{-i\mu(x-ia)}}{2i} \frac{e^{i\mu(x+ia)}-e^{-i\mu(x+ia)}}{2i} \\
&= \frac{e^{2\mu a}+e^{-2\mu a}}{4}- \frac{e^{2i\mu x}+e^{-2i\mu x}}{4} \\
&= \frac{1}{2} (\cosh(2\mu a)-\cos(2\mu x)) = \frac{1}{2}(\cosh(2\mu a)-1) + \sin^2(\mu x)
\end{align*}
By a slight abuse of notation let us rename the quantity $(1/2) (\cosh(2\mu a)-1)$ to be our new $a>0$. We generally think of $a$ as being small, though our estimates later will be true for arbitrary positive $a$.
Note that the new $a$ has dimension of $length^2.$ Combining the above calculations, our velocity $u$ can be now written as
\begin{align}
\label{vel}
u(x) =\frac{1}{2\pi} \int_0^L\omega(y)\left(\log|\sin^2[\mu(x-y)]|-\log|\sin^2[\mu(x-y)]+a|\right)  \, dy.
\end{align}

The main result of this section is the following
\begin{theorem}
\label{periodic}
There exist initial data with mean zero vorticity such that solutions to $\eqref{model1}$ and $\eqref{model2}$, with velocity given by $\eqref{vel}$ blow up in finite time. That is, there exists a time $T^*$ such that we have $\eqref{BKM}$.
\end{theorem}

\bigskip
We will consider the following type of initial data:
\begin{itemize}
\item $\theta_{0x}, \omega_0$ smooth, odd, periodic with period $L$
\item $\theta_{0x}, \omega_0 \ge 0$ on $[0,\frac{1}{2}L]$.
\item $\theta_0(0)=0$
\item $\|\theta_0\|_\infty \le M$
\end{itemize}
This can be visualized as follows:

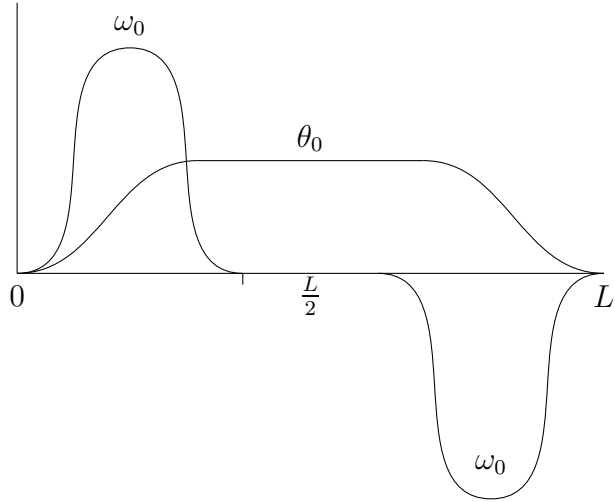
\begin{center}
\begin{tikzpicture}
\node at (0,-.3) {$0$};
\draw (0,0) -- (0,3.6);
\draw (0,0) -- (7.8,0);
\draw (0,0) to [out=0, in=180] (1.5, 3) to [out=0, in=180] (3,0);
\node at (1.5,3.3) {$\omega_0$};

\node at (3,-.3) {};

\draw (3,-.15) -- (3,0);
\draw (0,0) to [out=0, in=180] (2.4,1.5);
\draw (2.4,1.5) -- (5.4,1.5);
\node at (3.9,-.3) {$\frac{L}{2}$};
\node at (3.9,1.8) {$\theta_0$};
\draw (5.4,1.5) to [out=0, in=180] (7.8,0);
\draw (4.8,0) to [out=0, in=180] (6.3,-3) to [out=0,in=180] (7.8,0);
\node at (6.3,-2.55) {$\omega_0$};
\node at (7.8,-.3) {$L$};
\end{tikzpicture}
\end{center}

\noindent
From {\bf Proposition \ref{lwp}} one has the local well-posedness for our system($\eqref{model1}$$\eqref{model2}$$\eqref{vel}$). By local well posedness (in particular, uniqueness) and the transport structure of the system, all the above properties for our choice of initial data will be propagated in time up until possible blow up time.

The proof of singularity formation will follow by contradiction. The overall plan of the proof is based on finding appropriate functional of the solutions that blows up in finite time and goes back at least to the classical blow up argument in the nonlinear Schr\"odinger equation (see e.g. \cite{glassey1977blowing}). The motivation for the choice of initial data above is the following possible blow up scenario: we will have $u\le 0$ on $[0,L/2]$ and so $\theta$ will be pushed towards the origin by the flow 
increasing its derivative. This also causes $\omega$ to be pushed towards the origin while increasing its $L^\infty$ norm until there is velocity gradient blow up at the origin. The argument is similar in spirit to \cite{CCF} where the authors consider the quantity
$$
\int_0^{x_0}\frac{\omega(x,t)}{x}dx.
$$
Due to the periodic structure, the more natural quantity to monitor is, similarly to \cite{sixAuthors},
$$
\int_0^{\frac{L}{2}}\theta(x,t)\cot(\mu x)dx.
$$
Since $x=0$ is the stagnant point of the flow for all times while solution remains smooth, and since $\theta_0(0)=0,$ blow up of the above integral implies loss of regularity of the solution.
\\

We begin with derivation of some useful estimates for $u(x).$ 
Using that, due to our symmetry assumptions, our initial data is also odd with respect to $x=\frac{L}{2},$ we can write $u$ as
\begin{align*}
u(x) &=\frac{1}{\pi} \left[ \int_0^{L/2}+\int_{L/2}^L\right] \omega(y)\left(\log|\sin^2[\mu(x-y)]|-\log|\sin^2[\mu(x-y)]+a|\right)  \, dy \\
&= \frac{1}{\pi} \int_0^{L/2}\left( \log \left| \frac{\sin^2\mu (x-y)}{\sin^2\mu (x+y)}\right| + \log \left| \frac{\sin^2\mu (x+y)+a}{\sin^2\mu (x-y)+a}\right| \right)\omega(y)\, dy.
\end{align*}
Define
\begin{align}\label{Fper}
F(x,y,a) = \frac{\tan \mu y}{\tan \mu x} \left( \log \left| \frac{\sin^2\mu (x-y)}{\sin^2\mu (x+y)}\right| + \log \left| \frac{\sin^2\mu (x+y)+a}{\sin^2\mu (x-y)+a}\right| \right).
\end{align}
Then the Biot-Savart law \eqref{vel} can be written in the following form, which will be handy in the proof:
\begin{align}
\label{eq:general}
u(x)\cot(\mu x)=\frac{1}{\pi} \int_0^{L/2} F(x,y,a) \omega(y)\cot(\mu y)\, dy
\end{align}

The majority of this section will be devoted to establishing properties of $F$ that will allow for a proof of finite time blow up analogous
to the one for $HL$ model in \cite{sixAuthors}. These properties are contained in the following lemma.

\begin{lemma}
\label{keylemma}

{\bf (a)} There exists a positive constant $C$ depending on $a$ such that $F(x,y,a)\le -C < 0$ for $0<x<y< L/2$.

\smallskip
{\bf (b)} For any $0<y<x<\frac{L}{2}$, $F(x,y,a)$ is increasing in $x$.

\smallskip
{\bf (c)} For any $0<x,y<\frac{L}{2}$, $\cot(\mu y)(\partial_xF)(x,y,a)+\cot(\mu x)(\partial_x F)(y,x,a)$ is positive.

\end{lemma}

\medskip
\noindent
Note that $F$ is not symmetric in $x$ and $y$. Define
 $$
 K(x,y)=\frac{\tan \mu y}{\tan \mu x}  \log \left| \frac{\sin\mu (x+y)}{\sin\mu (x-y)}\right|,
 $$
 then
\begin{align}\label{Fx1}
F(x,y,a)=-2K(x,y)+ \frac{\tan \mu y}{\tan \mu x}\log \left| \frac{\sin^2\mu (x+y)+a}{\sin^2\mu (x-y)+a}\right|.
\end{align}

The term $K(x,y)$ arises from the original HL model and one can view it as the main contributor from $F$ towards the blow up. In order to show Lemma \ref{keylemma}, we first need the following technical lemma for $K(x,y)$.
\begin{lemma}
\label{HLlemma}
For simplicity, let us write $K(x,y)$ in the following form:
\begin{align}
K(x,y) = s\log \left| \frac{s+1}{s-1}\right|, \quad with \quad s= \frac{\tan(\mu y)}{\tan(\mu x)}.
\end{align}
Then it has the following properties:

\medskip
{\bf (a)} $K(x,y)\ge 0$ for all $x,y\in(0, \frac{1}{2}L)$ with $x\ne y$

\medskip
{\bf (b)} $K(x,y) \geq 2$ and $K_x(x,y) \geq 0$ for all $0<x<y<\frac{1}{2}L$

\medskip
{\bf (c)} $K(x,y) \geq 2s^2$ and $K_x(x,y) \geq 0$ for all $0<y<x<\frac{1}{2}L$

\end{lemma}
The detailed proof of Lemma \ref{HLlemma} can be found in \cite{sixAuthors}, Lemma 4.1.
\medskip
\noindent

\noindent
\begin{proof} [Proof of Lemma \eqref{keylemma}(a)]

\smallskip
First, it is easy to see that $F$ is non-positive. Indeed
\begin{align}\label{nonneg}
\left|\frac{\sin^2\mu (x-y)}{\sin^2\mu (x+y)}\right|\left|\frac{\sin^2\mu (x+y)+a}{\sin^2\mu (x-y)+a}\right| = \left| \frac{1+ \frac{a}{\sin^2 \mu(x+y)}}{1+ \frac{a}{\sin^2 \mu(x-y)}}\right| \le 1
\end{align}
because $\sin^2 \mu(x-y) \le \sin^2 \mu(x+y)$ if $x,y \in [0,L/2].$

\smallskip
For the better upper bound, we first consider the region $0<x<y<L/4$. For the region $L/4<x<y<L/2$, if we take $x^*=\frac{L}{2}-x$, $y^*=\frac{L}{2}-y$, then $0<y^*<x^*<L/4$, and relabelling of the variables
brings the kernel to the original form. This means the argument for this region would follow from that for the region $0<x<y<L/4$. 
We divide our estimate of this region into four separate cases. Let $a^*=\mbox{min}\{a,\frac{1}{16}\}$.

\medskip
\noindent
\ul{{\bf Case 1: $\frac{\sqrt{a^*}}{\pi}L=\frac{\sqrt{a^*}}{\mu }<x<y<L/4$}}

\medskip
In this domain we have $\sin\mu y>\sin\mu x>\frac{\sin(\frac{\pi}{4})}{\frac{\pi}{4}}\mu x> \frac{1}{\sqrt{2}}\mu x>\frac{1}{\sqrt{2}}\sqrt{a^*}$, $\cos\mu x>\cos\mu y >\frac{1}{\sqrt{2}}$, hence
$$
\sin^2\mu (x-y)= \sin^2\mu(x+y)-4\sin\mu x\sin\mu y\cos\mu x\cos\mu y< \sin^2\mu (x+y)- a^*,
$$
so
\begin{align}\label{rest}
F(x,y,a)\leq \log \left| \frac{\sin^2\mu (x-y)}{\sin^2\mu (x+y)}\right| + \log \left| \frac{\sin^2\mu (x+y)+a^*}{\sin^2\mu (x-y)+a^*}\right| &= \log \left| \frac{1+\frac{a^*}{\sin^2\mu (x+y)}}{1+\frac{a^*}{\sin^2\mu (x-y)}} \right|\\ &\le \log \left| \frac{1+\frac{a^*}{\sin^2\mu (x+y)}}{1+\frac{a^*}{\sin^2\mu (x+y)-a^*}} \right| \le -C_0(a)<0
\end{align}
where $C_0(a)$ is a positive constant independent of $x,y$. In the last step we use the fact that the function $\ds  \left(1+\frac{a^*}{z}\right)\left(1+\frac{a^*}{z- a^*}\right)^{-1}=1-\frac{(a^*)^2}{z^2}$ is increasing in $z$ for $ a^*< z<1$ and fixed $a^*$.

\medskip
\noindent
\ul{{\bf Case 2: $0<x<y<\frac{\sqrt{a^*}}{\mu }<L/4$} }

\medskip
From Lemma \ref{HLlemma} (b), we know
\begin{equation}\label{Kest12}
-4\geq -2K(x,y)=
\frac{\tan \mu y}{\tan \mu x} \log \left| \frac{\sin^2\mu (x-y)}{\sin^2\mu (x+y)}\right|.
\end{equation}
so if we can show the contribution from the other part of $F(x,y,a)$ is bounded above by some constant less than $4$, we are done. Expanding, we have that second term in $\eqref{Fx1}$ is equal to
\begin{align}
\label{eq:term}
\frac{\tan \mu y}{\tan \mu x} \log \left| \frac{\sin^2 \mu x\cos^2 \mu y+ 2\sin \mu x\cos \mu y \sin \mu y \cos \mu x + \sin^2 \mu y\cos^2 \mu x + a}{\sin^2 \mu x\cos^2 \mu y - 2\sin \mu x\cos\mu  y\sin \mu y\cos \mu x+ \sin^2 \mu y\cos^2 \mu x+ a}\right|.
\end{align}
Since $0<y<\frac{\sqrt{a^*}}{\mu }\leq \frac{\sqrt{a}}{\mu}$, we know $\sin^2 \mu y\cos^2\mu  x<\sin^2\sqrt{a}\cdot 1< a$. Then we have that $\eqref{eq:term}$ is bounded above by
\begin{align}\label{logfunction1}
\frac{\tan \mu y}{\tan \mu x} \log \left| \frac{\sin^2 \mu x\cos^2 \mu y+ 2\sin \mu x\cos \mu y \sin \mu y \cos \mu x + 2\sin^2 \mu y\cos^2 \mu x}{\sin^2 \mu x\cos^2\mu  y - 2\sin \mu x\cos \mu y\sin \mu y\cos \mu x+ 2\sin^2 \mu y\cos^2 \mu x}\right|= s\log \left| \frac{2s+\frac{1}{s}+2}{2s+\frac{1}{s}-2}\right|
\end{align}
where $\ds s =\frac{\tan \mu y}{\tan\mu  x}$. As a function of $s$, by direct calculation we find the derivative of the right hand side of (\ref{logfunction1}) is
\begin{equation}\label{del1}
\frac{4s-8s^3}{1+4s^4}+\log\left| \frac{2s+\frac{1}{s}+2}{2s+\frac{1}{s}-2}\right|.
\end{equation}
By taking the derivative of (\ref{del1}), we find the second derivative of (\ref{logfunction1}) is
$$
-\frac{8(4s^4+4s^2-1)}{(4s^4+1)^2},
$$
which is negative for $s \geq 1$. And we know that
$$
\lim_{s\to\infty}\left(\frac{4s-8s^3}{1+4s^4}+\log\left| \frac{2s+\frac{1}{s}+2}{2s+\frac{1}{s}-2}\right|\right)=0
$$
which means the right hand side of (\ref{logfunction1}) is increasing in $s$ for $s>1$ and $$\ds \lim_{s\to \infty} s\log \left| \frac{2s+\frac{1}{s}+2}{2s+\frac{1}{s}-2}\right|=2.$$

\medskip
\noindent
\ul{{\bf Case 3: $\frac{\sqrt{a^*}}{2\mu }<x<\frac{\sqrt{a^*}}{\mu }<y< L/4$}}

\medskip
In this case, because we know that $x$ is bounded away from zero, we have $\ds s= \frac{\tan \mu y}{\tan \mu x}\le C_1(a)$ for some constant depending on $a$. Also, $\cos^2 \mu y\sin^2 \mu x\leq 1\cdot \sin^2\sqrt{a}\le a$. Then $\eqref{eq:term}$ is bounded above by
\begin{equation}\label{logfunction2}
s \log \left| \frac{s+2+\frac{2}{s}}{s-2+\frac{2}{s}}\right|.
\end{equation}
Similarly to the previous case, the second derivative of (\ref{logfunction2}) is negative for $s>1$ and the limit of the first derivative of (\ref{logfunction2}) as $s$ goes to infinity is zero, which means (\ref{logfunction2}) monotonically (while $s \geq 1$) increases to $4$ as $s\to\infty$. However, since $s$ is bounded above, the expression \eqref{logfunction2} can be bounded by some constant $C_2(a)$ which is strictly less than $4$.
On the other hand, note that \eqref{Kest12} still applies.

\medskip
\noindent
\ul{{\bf Case 4: $0<x<\frac{\sqrt{a^*}}{2\mu }<\frac{\sqrt{a^*}}{\mu }<y<L/4$}}

\medskip
On the set $A=\{(x,y): 0\le x\le \frac{\sqrt{a^*}}{2\mu }, \frac{\sqrt{a^*}}{\mu }\le y\le L/4\}$, $F(x,y,a)$ is a continuous negative function (since $|x-y|$ has a positive lower bound and points where $x=0$ are removable singularities). Since $F\neq 0$ on $A$ and $A$ is compact, $F$ achieves a maximum $C_3(a)$ which is strictly less than $0$.

\medskip
This completes the analysis for the region $0<x<y<L/4$, and therefore for the region $L/4 < x< y < L/2$ by symmetry considerations. Now, we are left the domain $0<x<L/4<y<L/2$.

\medskip
This case is simpler and the analysis is divided in the following two cases. First, suppose $0<L/8<x<L/4<y<3L/8<L/4$ Then $\frac{3\pi}{8}<\mu(x+y)< \frac{5\pi}{8}$ and $0<\mu(y-x)<\frac{\pi}{4}$ so there exists $\epsilon>0$ such that $\sin^2 \mu(x+y)\ge \frac{1}{2}+\epsilon$. However, $\sin^2 \mu(x-y) < \frac{1}{2}$. From this, we get $\sin^2 \mu(x+y)-\sin^2\mu(x-y)\geq \epsilon^*$ for some constant $\epsilon^*$, which means (\ref{rest}) follows if we replace the $a^*$ by $\epsilon^*$. Then we get the desired estimate. If $x$ and $y$ are not in this region, there exists a constant $c>0$ such that $y-x>c>0$, then again by the same argument as in the {\bf Case 4} and we get the desired inequality.

\smallskip
\noindent
This completes the proof of (a).
\end{proof}

\medskip
\begin{proof}[{\bf Proof of \ref{keylemma}(b)}]
We compute directly and get
\begin{align*}
\cot(\mu y)(\partial_xF)(x,y,a)&=-\mu \csc^2(\mu x)\left(\log\left(\frac{\sin^2\mu (x-y)}{\sin^2\mu (x+y)}\right)+
\log\left(\frac{\sin^2\mu (x+y)+a}{\sin^2\mu (x-y)+a}\right)\right)\\
&+\mu \cot(\mu x)\left[\frac{2\sin\mu (x-y)\cos\mu (x-y)}{\sin^2\mu (x-y)}-\frac{2\sin\mu (x-y)\cos\mu (x-y)}{\sin^2\mu (x-y)+a}\right]\\
&-\mu \cot(\mu x)\left[\frac{2\sin\mu (x+y)\cos\mu (x+y)}{\sin^2\mu (x+y)}-\frac{2\sin\mu (x+y)\cos\mu (x+y)}{\sin^2\mu (x+y)+a}\right]\\
&=-\mu \csc^2(\mu x)\left(\log\left(\frac{\sin^2\mu (x-y)}{\sin^2\mu (x+y)}\right)+\log\left(\frac{\sin^2\mu (x+y)+a}{\sin^2\mu (x-y)+a}\right)\right)\\
&+\mu \cot(\mu x)\left[\frac{2a\sin\mu (x-y)\cos\mu (x-y)}{\sin^2\mu (x-y)(\sin^2\mu (x-y)+a)}-\frac{2a\sin\mu (x+y)\cos\mu (x+y)}{\sin^2\mu (x+y)(\sin^2\mu (x+y)+a)}\right]\\
&=I+II.\\
\end{align*}

\noindent
The term $I$, by the same calculation as (\ref{nonneg}), is positive. The term $II$, when $x>y$, can be expressed as
$$
\cot(\mu x)(g(x-y)-g(x+y)),
$$
where $g(t)=\frac{\cos(\mu t)}{\sin(\mu t)(\sin^2(\mu t)+a)}$. It is easy to see that whenever $0<y<x<\frac{L}{2}$, $\cos\mu (x-y)\geq \cos\mu (x+y)$, $0 \leq \sin\mu (x-y)\leq \sin\mu (x+y)$. This means that $g(x-y)\geq g(x+y)$, which implies $II\geq 0$. This completes the proof of (b).
\end{proof}

\begin{proof}[{\bf Proof of \ref{keylemma}(c)}]

Now, for the final part of the lemma. First of all, we set
\begin{align*}
G(x,y,a)&=\cot(\mu y)(\partial_xF)(x,y,a)+\cot(\mu x)(\partial_x F)(y,x,a)\\
& = -\mu (\csc^2(\mu x)+\csc^2(\mu y))\left[\log\left(\frac{\sin^2\mu (x-y)}{\sin^2\mu (x+y)}\right)+
\log\left(\frac{\sin^2\mu (x+y)+a}{\sin^2\mu (x-y)+a}\right)\right]\\
&+\mu (\cot(\mu x)-\cot(\mu y))\frac{2a\sin\mu (x-y)\cos\mu (x-y)}{\sin^2\mu (x-y)(\sin^2\mu (x-y)+a)}\\
&-\mu (\cot(\mu x)+\cot(\mu y))\frac{2a\sin\mu (x+y)\cos\mu (x+y)}{\sin^2\mu (x+y)(\sin^2\mu (x+y)+a)}.\\
&=-\mu (\cot^2(\mu x)+\cot^2(\mu y)+2)\left[\log\left(\frac{\sin^2\mu (x-y)}{\sin^2\mu (x+y)}\right)+
\log\left(\frac{\sin^2\mu (x+y)+a}{\sin^2\mu (x-y)+a}\right)\right]\\
&-\mu \frac{2a\cos\mu (x-y)}{(\sin^2\mu (x-y)+a)\sin(\mu x)\sin(\mu y)}-\mu \frac{2a\cos\mu (x+y)}{(\sin^2\mu (x+y)+a)\sin(\mu x)\sin(\mu y)}\\
&=-\mu (\cot^2(\mu x)+\cot^2(\mu y)+2)\left[\log\left(\frac{\sin^2\mu (x-y)}{\sin^2\mu (x+y)}\right)+
\log\left(\frac{\sin^2\mu (x+y)+a}{\sin^2\mu (x-y)+a}\right)\right]\\
&-2\mu \cot(\mu x)\cot(\mu y)\left[\frac{a}{\sin^2\mu (x-y)+a}+\frac{a}{\sin^2\mu (x+y)+a}\right]\\
&-2\mu \left[\frac{a}{\sin^2\mu (x-y)+a}-\frac{a}{\sin^2\mu (x+y)+a}\right]
\end{align*}
Now our aim is to prove the positivity of $G(x,y,a)$. Notice that when $a=0$, $G(x,y,a)=0$,
as a consequence, to prove the positivity of $G(x,y,a)$, the only thing we need to show is that this function is increasing in $a$ for any $x,y$ in the domain. On the other hand,
\begin{align*}
\frac{1}{\mu }\partial_aG(x,y,a)&=(\cot^2(\mu x)+\cot^2(\mu y)+2)\left[\frac{1}{\sin^2\mu (x-y)+a}-\frac{1}{\sin^2\mu (x+y)+a}\right]\\
&-2\cot(\mu x)\cot(\mu y)\left[\frac{\sin^2\mu (x-y)}{(\sin^2\mu (x-y)+a)^2}+\frac{\sin^2\mu (x+y)}{(\sin^2\mu (x-y)+a)^2}\right]\\
&-2\left[\frac{\sin^2\mu (x-y)}{(\sin^2\mu (x-y)+a)^2}-\frac{\sin^2\mu (x+y)}{(\sin^2\mu (x+y)+a)^2}\right]\\
&=(\cot^2(\mu x)+\cot^2(\mu y)+2)\frac{\sin^2\mu (x+y)-\sin^2\mu (x-y)}{(\sin^2\mu (x-y)+a)(\sin^2\mu (x+y)+a)}\\
&-2\cot(\mu x)\cot(\mu y)\left[\frac{\sin^2\mu (x-y)}{(\sin^2\mu (x-y)+a)^2}+\frac{\sin^2\mu (x+y)}{(\sin^2\mu (x-y)+a)^2}\right]\\
&-2\left[\frac{\sin^2\mu (x-y)}{(\sin^2\mu (x-y)+a)^2}-\frac{\sin^2\mu (x+y)}{(\sin^2\mu (x+y)+a)^2}\right].\\
\end{align*}
Therefore,
\begin{align*}
&\frac{1}{\mu }((\sin^2\mu (x-y)+a)(\sin^2\mu (x+y)+a))^2\partial_aG(x,y,a)\\
&=(\cot^2(\mu x)+\cot^2(\mu y)+2)(\sin^2\mu (x+y)-\sin^2\mu (x-y))(\sin^2\mu (x-y)+a)(\sin^2\mu (x+y)+a)\\
&-2\cot(\mu x)\cot(\mu y)\left[\sin^2\mu (x-y)(\sin^2\mu (x+y)+a)^2+\sin^2\mu (x+y)(\sin^2\mu (x-y)+a)^2\right]\\
&-2\left[\sin^2\mu (x-y)(\sin^2\mu (x+y)+a)^2-\sin^2\mu (x+y)(\sin^2\mu (x-y)+a)^2\right].
\end{align*}
It is easy to see that this is a quadratic polynomial in $a$ of the form $A_2a^2+A_1a+A_0$. We will explicitly compute $A_2,A_1,$ and $A_0$ and show each term is non-negative. For the second order term we get
\begin{align*}
A_2&=(\cot^2(\mu x)+\cot^2(\mu y)+2)(\sin^2\mu (x-y)-\sin^2\mu (x+y))\\
&-2\cot(\mu x)\cot(\mu y)[\sin^2\mu (x-y)+\sin^2\mu (x+y)]\\
&-2[\sin^2\mu (x-y)-\sin^2\mu (x+y)].\\
&=(\cot^2(\mu x)+\cot^2(\mu y))(\sin^2\mu (x+y)-\sin^2\mu (x-y))\\
&-2\cot(\mu x)\cot(\mu y)[\sin^2\mu (x-y)+\sin^2\mu (x+y)].
\end{align*}
This means
\begin{align*}
\tan(\mu x)\tan(\mu y)A_2 = &(\frac{\tan(\mu x)}{\tan(\mu y)}+\frac{\tan(\mu y)}{\tan(\mu x)})(\sin^2\mu (x+y)-\sin^2\mu (x-y))\\
&-2[\sin^2\mu (x-y)+\sin^2\mu (x+y)].
\end{align*}
If we set $\frac{\tan(\mu x)}{\tan(\mu y)}=s$, we get
$$
\frac{\tan(\mu x)\tan(\mu y)}{\cos(\mu y)\cos(\mu x)\sin(\mu y)\sin(\mu x)}A_2=(s+\frac{1}{s})\cdot 4-2[2\cdot(s+\frac{1}{s})]=0.
$$
This means as long as $0<x, y<\frac{L}{2}$, $A_2= 0$. Similarly, for coefficient of the first order term $A_1$, we have
\begin{align*}
A_1&=(\cot^2(\mu x)+\cot^2(\mu y)+2)(\sin^2\mu (x+y)-\sin^2\mu (x-y))(\sin^2\mu (x+y)+\sin^2\mu (x-y))\\
&-2\cot(\mu x)\cot(\mu y)[2\sin^2\mu (x-y)\sin^2\mu (x+y)+2\sin^2\mu (x+y)\sin^2\mu (x-y)]\\
&-2[2\sin^2\mu (x-y)\sin^2\mu (x+y)-2\sin^2\mu (x+y)\sin^2\mu (x-y)]\\
&\geq (\cot^2(\mu x)+\cot^2(\mu y)+2)[\sin^4\mu (x+y)-\sin^4\mu (x-y)]\\
&-8\cot(\mu x)\cot(\mu y)[\sin^2\mu (x-y)\sin^2\mu (x+y)].
\end{align*}
Again, by setting $\frac{\tan(\mu x)}{\tan(\mu y)}=s$, we get
$$
\frac{\tan(\mu x)\tan(\mu y)}{\cos(\mu x)\cos(\mu y)\sin(\mu x)\sin(\mu y)}A_1\geq (s+\frac{1}{s})\cdot 4\cdot 2(s+\frac{1}{s})-8(s+\frac{1}{s}-2)(s+\frac{1}{s}+2)\geq 32.
$$
Lastly, for the coefficient of the constant term $A_0$, we have
\begin{align*}
A_0&= (\cot^2(\mu x)+\cot^2(\mu y))(\sin^2\mu (x+y)-\sin^2\mu (x-y))\sin^2\mu (x+y)\sin^2\mu (x-y)\\
&-2\cot(\mu x)\cot(\mu y)[\sin^2\mu (x-y)\sin^2\mu (x+y)(\sin^2\mu (x+y)+\sin^2\mu (x-y))]\\
&-2 \sin^2\mu (x-y)\sin^2\mu (x+y)[\sin^2\mu (x+y)-\sin^2\mu (x-y)]\\
&=(\cot^2(\mu x)+\cot^2(\mu y))(\sin^2\mu (x+y)-\sin^2\mu (x-y))\sin^2\mu (x+y)\sin^2\mu (x-y)\\
&-2\cot(\mu x)\cot(\mu y)\sin^2\mu (x-y)\sin^2\mu (x+y)[\sin^2\mu (x+y)+\sin^2\mu (x-y)].
\end{align*}
Setting again $s=\frac{\tan(\mu x)}{\tan(\mu y)},$ after computation we have
$$
\frac{\tan(\mu x)\tan(\mu y)}{\sin^2\mu (x-y)\sin^2\mu (x+y)\cos(\mu x)\cos(\mu y)\sin(\mu x)\sin(\mu y)}A_0=(s+\frac{1}{s})\cdot 4-2\cdot(2s+\frac{2}{s})=0.
$$
In all, we have $\partial_a G(x,y,a)\geq 0$ for $0<x,y<\frac{L}{2}$. This completes the proof.
\end{proof}
\begin{remark}
One may notice that when $a\rightarrow \infty$, $\frac{1}{\mu}G(x,y,a)$ tends to
\begin{align}
-(\cot^2(\mu x)+\cot^2(\mu y)+2)\left[\log\left(\frac{\sin^2\mu (x-y)}{\sin^2\mu (x+y)}\right)\right]-4\cot(\mu x)\cot(\mu y).
\end{align}
The positivity of this quantity is also proved by Lemma 4.2 in \cite{sixAuthors}, in which the authors use technical trigonometric inequalities. Our proof of the above lemma provides another approach to estimating this quantity.
\end{remark}

\bigskip
With these lemmas at our disposal, we are ready to prove finite-time blow up.

\noindent
{\bf Proof of Theorem \ref{periodic}}.

\medskip
Suppose we have a global smooth solution. We will show blow up of the following quantity:
\begin{align*}
I(t) := \int_0^{L/2} \theta(x,t)\cot(\mu x)\, dx.
\end{align*}
thereby arriving at a contradiction since
\begin{align*}
|I(t)| \le C\|\theta_x(\cdot, t)\|_{L^\infty} \leq C\|\theta_{0x}\|_{L^\infty} \exp \left( \int_0^t\|u_x(\cdot,s)\|_{L^\infty}\, ds\right).
\end{align*}
If $I$ were to become infinite in finite time, we would be  able to use Beale-Kato-Majda type condition for the system as stated in equation \eqref{BKM} from which we can conclude finite time blow up.

We first compute the derivative of $I(t)$:
\begin{align*}
\frac{d}{dt} I(t) = -\frac{1}{\pi} \int_0^{L/2} \theta_x(x,t)\int_0^{L/2} \omega(y,t) \cot(\mu y) F(x,y,a)\, dy\, dx.
\end{align*}
By the negativity of $F$ and part (a) of the lemma, the expression above is bounded below by
\begin{align*}
\frac{C}{\pi} \int_0^{L/2} \theta_x(x,t)\int_x^{L/2} \omega(y,t)\cot(\mu y)\, dy\, dx = \frac{C}{\pi} \int_0^{L/2} \theta(y,t) \omega(y,t) \cot(\mu y)\, dy := CJ(t)
\end{align*}
(where $J(t)=\frac{2}{\pi} \int_0^{L/2} \theta(x,t) \omega(x,t) \cot(\mu x)\, dx$). Then
\begin{align}\label{eJ}
\frac{d}{dt} (J(t)) = \frac{C}{\pi} \int_0^{L/2} \theta(x,t) \omega(x,t) \left( u(x,t)\cot(\mu x)\right)_x\, dx + \frac{C\mu}{2\pi} \int_0^{L/2} \theta^2(x,t) \csc^2(\mu x)\, dx.
\end{align}
By Cauchy-Schwarz inequality, the second integral is bounded below by $\frac{C}{L^2} I(t)^2$ for some constant $C$. The first integral is given by
\begin{align}\label{eP}
\frac{C}{\pi} \int_0^{L/2} \theta_y(y) \left[ \int_y^{L/2} \omega(x) \left( u(x)\cot (\mu x)\right)_x \, dx\right]\, dy
\end{align}
Observe that since $\theta$ is non-decreasing on $[0,L/2]$, the expression $\eqref{eP}$ is positive if we can show the integral in the brackets is positive as well. This is our next task. For $x,y\in[0,\frac{1}{2}L]$, $\omega(x)$ can be decomposed as
\begin{align*}
\omega(x) =\omega(x)\chi_{[0,y]} (x) +\omega(x) \chi_{[y,\frac{1}{2}L]} (x) =: \omega_\ell (x)+ \omega_r(x).
\end{align*}
Then we can decompose the integral:
\begin{align*}
\int_y^{L/2} \omega(x) [u(x)\cot(\mu x)]_x\, dx &= \frac{1}{\pi} \int_0^{L/2} \omega_r(x) \int_0^{L/2} \omega_\ell(y)\cot(\mu y)(\partial_x F)(x,y,a)\, dy\, dx \\ &\quad + \frac{1}{\pi} \int_0^{L/2} \omega_r(x)\int_0^{L/2} \omega_r(y) \cot(\mu y) (\partial_x F)(x,y,a)\, dy\, dx
\end{align*}
By positivity of $\omega$ on $[0,\frac{1}{2}L]$ and part (b) of the key lemma, the first integral is non-negative. By using symmetry, the second integral is equal to

\begin{align*}
\frac{1}{2\pi} \int_0^{L/2}\int_0^{L/2} \omega_r(x)\omega_r(y) G(x,y,a)\, dy\, dx
\end{align*}
where as before $G(x,y,a)=\cot(\mu y)(\partial_xF)(x,y,a)+\cot(\mu x)(\partial_x F)(y,x,a)$. However, by part (c) of the lemma, this is positive. Together with (\ref{eJ}) and (\ref{eP}) we have:
\begin{equation}
\frac{d^2}{dt^2}I\geq CI^2,
\end{equation}
for some constant $C$. To close the proof, we only need the following lemma:
\begin{lemma}\label{ode}
Suppose $I(t)$ solves the following initial value problem:
\begin{equation}
\frac{d}{dt}I(t)\geq C \int_0^t I^2(s)ds,\,\,\, I(0)=I_0.
\end{equation}
Then there exists $T=T(C,I_0)$ so that $\lim_{t\rightarrow T}I(t)=\infty$.

Moreover, for fixed C and any $\epsilon>0$, there is an $A>0$ (depending on $C,\epsilon$), so that for any $I_0\geq A$, the blow up time $T<\epsilon$.
\end{lemma}
The proof of this lemma is straightforward, and one can also find a sketch of the proof in \cite{sixAuthors}.

\section{Stability of Blow Up with Respect to Perturbations}\label{pertblowup}
In this section, we consider our system $\eqref{model1}$ and $\eqref{model2}$ but with a Biot-Savart law which is a perturbation of the Hou-Lou kernel. As before, we will work with periodic solutions with period $L$,
and assume that the vorticity is odd (this property will be conserved in time for the perturbations we consider). The velocity $u$ is given by the following choice of Biot-Savart law
\begin{align}\label{genper1}
u(x) &= \frac{1}{\pi} \int_0^L \left( \log|\sin[\mu(x-y)]| + f(x,y)\right) \omega(y)\, dy, \quad \mu:=\pi/L \\
&:= u_{HL}(x) + u_f(x)
\end{align}
where $f$ is a smooth function whose precise properties we will specify later. We view $f$ as a perturbation and we will  show solutions to the system (\ref{model1}) and (\ref{model2}) with \eqref{genper1}
can still blow up in finite time. As with the previous system \eqref{model1}, \eqref{model2}, \eqref{kern}, it is not hard to show that we will still have a local well-posedness result akin to Proposition \eqref{lwp}.
In particular, if $T^*$ is a maximal time of existence of a solution then we must have
\begin{align}\label{BKM1}
\lim_{t\nearrow T^*} \int_0^t \|u_x(\cdot,\tau)\|_{L^\infty}\, d\tau=\infty
\end{align}
We show below that such a time will exist for some initial data.

\begin{theorem}\label{genper}
Let $f\in C^2(\mathbb{R}^2)$, periodic with period $L$ with respect to both variables  and such that $f(x,y)=f(-x,-y)$ for all $x,y$. Then there exists initial data $\omega_0$, $\theta_0$  such that solutions of $\eqref{model1}$ and $\eqref{model2}$, with velocity given by (\ref{genper1}), blow up in finite time. Again, that means there exists a time $T^*$ such that we have (\ref{BKM1}).
\end{theorem}

\noindent
We will consider the following class of initial data:

\begin{itemize}
\item $\theta_{0x}, \omega_0$ smooth odd periodic with period $L$
\item $\theta_{0x}, \omega_0 \ge 0$ on $[0,\frac{1}{2}L]$.
\item $\theta_0(0)=0$
\item $(\supp \theta_{0x}\cup\supp \omega_0) \cap [0, \frac{1}{2} L] \subset [0,\epsilon]$
\item $\|\theta_0\|_\infty \le M$
\end{itemize}

\noindent
We will make the choice of specific $\epsilon$ below.
Observe that by the assumptions, $\omega_0$ and $\theta_{0x}$ are also odd with respect to $\frac{1}{2}L$.
By the following Lemma \ref{44}, we can choose $\epsilon$ sufficiently small so that the mass of $\omega$ near the origin gets closer to the origin leading to a scenario where blow up can be achieved.

\medskip
\noindent
\begin{remark} With the choice of $f(x,y)= \log\sqrt{\sin^2\mu(x-y)+a}$, we have the kernel from the previous section. However, in the previous section, we proved blow up for a larger class of initial data.
\end{remark}
\medskip
\noindent
\begin{lemma} \label{44}
With the initial data $\omega_0$ and $\theta_0$ as given above, we can choose $\epsilon_1$ sufficiently small so that for $\epsilon<\epsilon_1$, $u(x)< 0$ for $x\le \epsilon$ where $u$ is defined as \eqref{genper1}.
\end{lemma}

\begin{proof}
By periodicity and support property of $\omega$,
\begin{align*}
u(x) &= \frac{1}{\pi} \int_0^{L/2} \left( \log\left| \frac{\tan(\mu x)-\tan(\mu y)}{\tan(\mu x)+ \tan(\mu y)}\right|+ f(x,y)-f(x,-y)\right)\omega(y)\, dy \\
&=  \frac{1}{\pi} \int_0^{\epsilon} \left( \log\left| \frac{\tan(\mu x)-\tan(\mu y)}{\tan(\mu x)+ \tan(\mu y)}\right|+ f(x,y)-f(x,-y)\right)\omega(y)\, dy.
\end{align*}
By the mean value theorem, for $0\le y\le \epsilon$, $|f(x,y)-f(x,-y)|\le 2\epsilon \|f\|_{C^1}$. By the singularity of the $HL$ kernel when $x=y=0$, we can choose $\epsilon_1$ such that the expression in the parentheses is negative for $0<x,y\le \epsilon$.
\end{proof}

It follows that under our assumptions on the initial data, $\omega(x,t)$ and $\theta_x(x,t)$ are supported on $[0,\epsilon]$ for all times while regular solution exists.
\medskip
\noindent
We will also need the following lemma controlling the integral of $\omega$ over half the period.
\begin{lemma}
There exists $\epsilon_2>0$ such that for $\epsilon<\epsilon_2$, with $\omega_0$ and $\theta_0$ as chosen above, solutions of $\eqref{model1}$, $\eqref{model2}$, $\eqref{genper1}$ satisfy
$$
\int_0^{L/2} \omega(y,t) \, dy \le Mt.
$$
\end{lemma}

\begin{proof}
Integrating both sides of \eqref{model1} and integrating by parts we get
\begin{align*}
\int_0^{L/2} \omega_t(y,t)\, dy = \int_0^{L/2} u_x(y) \omega(y,t)\, dy + \int_0^{L/2} \theta_x(y,t)\, dy \le M+ \int_0^{L/2} u_x(y) \omega(y,t)\, dy
\end{align*}
If we can show the remaining integral on the right is negative, we are done. Due to our symmetry assumptions, the integral can be written as
\begin{align*}
\frac{1}{\pi} \int_0^{L/2} P.V.\int_0^{L/2} \left(\mu\cot[\mu(x-y)]-\mu\cot[\mu(x+y)] +f_x(x,y)-f_x(x,-y)\right) \omega(x,t)\omega(y,t)\, dy\, dx.
\end{align*}
By symmetry, the integral with $\cot[\mu(x-y)]$ is $0$ and using the support property of $\omega$, the above expression is equal to
\begin{align*}
\frac{1}{\pi} \int_0^\epsilon \int_0^\epsilon \left(-\cot[\mu(x+y)]+f_x(x,y)-f_x(x,-y)\right) \omega(x,t)\omega(y,t)\, dy\, dx
\end{align*}
Since $f$ is smooth and $\omega$ is positive, we can make $\epsilon_2$ small enough so that the kernel in the parentheses above in the integrand is negative.
\end{proof}

\medskip
\noindent
Now, so we can take advantage of our lemmas, we choose $\epsilon=\min\{\epsilon_1,\epsilon_2\}$ for the support of our initial data.

\bigskip
\noindent
\begin{proof}[Proof of Theorem \ref{genper}] Throughout, $C(f)$ will be a positive constant that only depends on $f$ and not $\omega_0$. We will show that
\begin{align}
I(t) := \int_0^{L/2} \theta(x,t)\cot(\mu x)\, dx
\end{align}
must blow up. Taking time derivative of $I$ and using Lemma \ref{HLlemma}, we get
\begin{align*}
\frac{d}{dt} I(t) &= -\int_0^{L/2} u(x) \theta_x(x)\cot (\mu x)\, dx \\
&= \frac{1}{\pi} \int_0^{L/2} \theta_x(x) \int_0^{L/2} \omega(y) \cot(\mu y) K(x,y)\, dy\, dx \\
&\quad + \int_0^{L/2} \theta_x(x) \left(u_f(x)\cot(\mu x)\right)\, dx \ge J(t) + \int_0^{L/2} \theta_x(x) \left(u_f(x) \cot(\mu x)\right)\, dx
\end{align*}
where, using the same notation as before,
$$
J(t) = \frac{2}{\pi} \int_0^{L/2} \theta(x)\omega(x)\cot(\mu x)\, dx
$$
Now, we would like to bound the extra term arising because of $f$. Since $f$ is smooth and $\omega$ is supported near the origin,
\begin{align*}
|u_f(x)\cot(\mu x)| =\left| \int_0^\epsilon \left[ \cot(\mu x)(f(x,y)-f(x,-y))\right] \omega(y)\,dy\right| \le C(f)  \cdot \left(\int_0^{L/2} \omega(y)dy\right).
\end{align*}

\noindent
Therefore, we have
\begin{align}
\label{eq:I}
\frac{d}{dt} I(t) \ge J(t)- C(f)M\left(\int_0^{L/2}\omega(y)dy \right)  \ge J(t) - C(f) M^2t
\end{align}

\noindent
Now, we derive a differential inequality for $J(t)$.
\begin{align*}
\frac{d}{dt} J(t) &= \frac{2}{\pi} \int_0^{L/2} -(\theta(x)\omega(x))_x u(x)\cot(\mu x) + \theta_x(x)\theta(x)\cot(\mu x)\, dx \\ &= \frac{2}{\pi}\int_0^{L/2} \theta(x)\omega(x) (u(x)\cot(\mu x))_x\, dx + \frac{\mu}{\pi} \int_0^{L/2} \theta^2(x)\csc^2(\mu x)\, dx
\end{align*}
As before, by Cauchy-Schwarz inequality the second integral is bounded below by $\ds \frac{2}{L^2} I(t)^2$. We split the first integral into two parts:
\begin{align*}
\frac{2}{\pi}\int_0^{L/2} \theta(x)\omega(x) (u_{HL}(x)\cot(\mu x))_x\, dx+ \frac{2}{\pi}\int_0^{L/2} \theta(x)\omega(x) (u_f(x)\cot(\mu x))_x\, dx.
\end{align*}
By the arguments in the proof of theorem \ref{periodic}, the first integral is positive. The second integral is equal to
\begin{align}
\label{eq:J1}
\frac{2}{\pi} \int_0^{L/2} \theta_y(y) \left[\int_y^{L/2} \omega(x)(u_f(x)\cot(\mu x))_x\, dx\right]\, dy
\end{align}
Using the smoothness, boundedness, and symmetries of $f$, we have
\begin{align}\label{c3}
|\partial_x(u_f(x)\cot(\mu x))| = \left|\int_0^\epsilon \partial_x \left[ \cot(\mu x)(f(x,y)-f(x,-y))\right] \omega(y)\,dy\right|
\end{align}
Now let $h(x,y)=\cot(\mu x)(f(x,y)-f(x,-y))$. Then it is easy to see that $h\in C^1$ when $f\in C^2$, which means that $|\partial_x h(x,y)|$ is bounded above. This implies that
 the right hand side of (\ref{c3}) can be bounded above by
$$
 C(f)  \cdot \left(\int_0^{L/2} \omega(y)dy\right).
$$
Inserting this estimate into \eqref{eq:J1}, and using monotonicity of $\theta,$ we get that \eqref{eq:J1} is bounded below by
\begin{align*}
- C(f)  M \left(\int_0^{L/2} \omega(y)dy\right)^2.
\end{align*}
Putting things together, we get
\begin{align}
\label{eq:J}
\frac{d}{dt} J(t) \ge \frac{2}{L^2} I(t)^2 - C(f)  M \left(\int_0^{L/2} \omega(y)dy\right)^2 \ge \frac{2}{L^2}I(t)^2 - C(f) M^3 t^2
\end{align}
Now, we will show that the differential inequalities we have established will lead to finite time blow up. By $\eqref{eq:I}$ and $\eqref{eq:J}$ we obtain
\begin{equation}\label{Ibound}
\begin{split}
\frac{d}{dt} I(t) & \ge \frac{2}{L^2} \int_0^t I^2(s)\, ds + J(0) - c(f) M^2 t- C(f) M^3 \frac{t^3}{3}\\
& \ge  \frac{2}{L^2} \int_0^t I^2(s)\, ds- c(f)  M^2 t- C(f) M^3 \frac{t^3}{3}.
\end{split}
\end{equation}
%
%
%
%
We claim that one can choose $I(0)$ large enough so that the effect of the negative terms is controlled. By a rather crude estimate we have
\begin{align*}
\frac{d}{dt} I(t)& \ge - c(f)  M^2 t- C(f) M^3 \frac{t^3}{3}.\\
\end{align*}
After integration, this implies
\begin{align}\label{initialI}
I(t) \ge I(0)-C(f) M^2\left( \frac{t^2}{2}+M\frac{t^4}{12}\right).
\end{align}
Now fix a time, say $1$. We will show that $I(0)$ can be chosen large enough so that $I(t)$ blows up before time $1$. Note that assuming $I(0)\geq C(f) M^2\left( \frac{1}{2}+M\frac{1}{12}\right)$, we have for $t \leq 1$,
\begin{align*}
\frac{1}{L^2} \int_0^{t} I^2(s)\, ds \ge\frac{t}{L^2} \left[I(0)-C(f) M^2\left( \frac{1}{2} +\frac{M}{12}\right)\right]^2
\end{align*}
Choose $I(0)$ so that
\begin{align}\label{estaux123}
I(0) \ge C(f)M^2\left(\frac{1}{2}+\frac{M}{12}\right) + L\sqrt{c(f)M^2+C(f)\frac{M^3}{3}}
\end{align}
Then, for $0 \leq t \leq 1$, with this choice of $I(0)$ and using \eqref{Ibound} and \eqref{estaux123}, we get
\begin{align*}
\frac{d}{dt} I(t) &\ge \frac{1}{L^2} \int_0^t I(s)^2\, ds+t\left(c(f)M^2+C(f)\frac{M^3}{3}\right)- c(f)  M^2 t- C(f) M^3 \frac{t^3}{3}\\
&\ge  \frac{1}{L^2} \int_0^t I(s)^2\, ds
\end{align*}
By perhaps making $I(0)$ a little larger, if needed, we can show $I(t)$ becomes infinite before time $1$ by Lemma \ref{ode}.
\end{proof}

\section{Appendix: Real Line Case}
One can also consider the model equation (\ref{model1}) and (\ref{model2}) with the law $\eqref{kern}$ for compactly supported data on $\mathbb{R}$.
we only outline main ideas and changes involved, leaving all details to the interested reader.
Without loss of generality we assume the domain of the initial data is $[-1,1]$. 
In this case, similar argument like in Section 2 can show that the corresponding modified Hou-Luo kernel will be
\begin{equation}\label{Fnon}
F(x,y,a)=\frac{y}{x}\left[\log\left(\frac{(x-y)^2}{(x+y)^2}\right)+\log\left(\frac{(x+y)^2+a}{(x-y)^2+a}\right)\right],
\end{equation}
for $a>0$.

The analogue of Lemma \ref{keylemma} will be the following:

\begin{lemma}

{\bf (a)} For any $a\neq 0$, there is a constant $C(a)> 0$ such that for any $0<x<y<1$, $F(x,y,a)\leq -C(a)$.

\medskip
{\bf (b)}For any $0<y<x<\infty$, $F(x,y,a)$ is increasing in $x$.

\medskip
{\bf (c)} For any $0<x,y<\infty$, $\frac{1}{y}(\partial_xF)(x,y,a)+\frac{1}{x}(\partial_x F)(y,x,a)$ is positive.

\end{lemma}
\begin{proof}
First it is easy to see that $F(x,y,a)$ is non-positive. For part (a), one can follow the similar but easier argument as in the proof of part (a) of Lemma \ref{keylemma}. Now let us prove part (b) and (c).

\smallskip
\noindent
{\bf Proof of (b)}

\medskip
By direct computation
\begin{align*}
\frac{1}{y}\partial_xF(x,y,a)&=-\frac{1}{x^2}\left[\log\left(\frac{(x-y)^2}{(x+y)^2}\right)+\log\left(\frac{(x+y)^2+a}{(x-y)^2+a}\right)\right]\\
&+\frac{1}{x}\left[\frac{2(x-y)}{(x-y)^2}-\frac{2(x-y)}{(x-y)^2+a}-\frac{2(x+y)}{(x+y)^2}+\frac{2(x+y)}{(x+y)^2+a}\right]\\
&=-\frac{1}{x^2}\left[\log\left(\frac{(x-y)^2}{(x+y)^2}\right)+\log\left(\frac{(x+y)^2+a}{(x-y)^2+a}\right)\right]\\
&+\frac{1}{x}\left[\frac{2a(x-y)}{(x-y)^2((x-y)^2+a)}-\frac{2a(x+y)}{(x+y)^2((x+y)^2+a)}\right]\\
&=I+II.\\
\end{align*}
The term $I$, by the same argument as in the proof of the periodic analog, is positive. For the term $II$, we have
$$
II=\frac{1}{x}(g(x-y)-g(x+y)),
$$
 where $g(t)=\frac{2a}{t(t^2+a)}$. It is easy to see that for $t>0$, $g(t)$ is decreasing in $t$, which means $II\geq 0$ whenever $0<y<x$.

\smallskip
\noindent
{\bf Proof of (c)}

\medskip
First of all, let us call our target function $G(x,y,a)$, which means
\begin{align*}
G(x,y,a)&=\frac{1}{y}(\partial_x F)(x,y,a)+\frac{1}{x}(\partial_x F)(y,x,a)\\
&=-\left(\frac{1}{x^2}+\frac{1}{y^2}\right)\left[\log\left(\frac{(x-y)^2}{(x+y)^2}\right)+\log\left(\frac{(x+y)^2+a}{(x-y)^2+a}\right)\right]\\
&+\left(\frac{1}{x}-\frac{1}{y}\right)\left(\frac{2a(x-y)}{(x-y)^2((x-y)^2+a)}\right)-\left(\frac{1}{y}+\frac{1}{x}\right)\left(\frac{2a(x+y)}{(x+y)^2((x+y)^2+a)}\right)\\
&=-\left(\frac{1}{x^2}+\frac{1}{y^2}\right)\left[\log\left(\frac{(x-y)^2}{(x+y)^2}\right)+\log\left(\frac{(x+y)^2+a}{(x-y)^2+a}\right)\right]\\
&-\frac{2a}{xy((x-y)^2+a)}-\frac{2a}{xy((x+y)^2+a)}.
\end{align*}
 Now our aim is to prove the positivity of $G(x,y,a)$. Notice that when $a=0$, $G(x,y,a)=0$, as a consequence, to prove the positivity of $G(x,y,a)$, the only thing we need to show is this function is increasing in $a$ for any $x,y$ in the domain. On the other hand,
\begin{align*}
\partial_a G(x,y,a)&=-\left(\frac{1}{x^2}+\frac{1}{y^2}\right)\left(\frac{1}{(x+y)^2+a}-\frac{1}{(x-y)^2+a}\right)\\
&-\frac{2}{xy}\left[\frac{(x-y)^2}{((x-y)^2+a)^2}+\frac{(x+y)^2}{((x+y)^2+a)^2}\right].
\end{align*}
As a conclusion,
\begin{align*}
&((x-y)^2+a)^2((x+y)^2+a)^2\partial_a G(x,y,a)\\
&=\left(\frac{1}{x^2}+\frac{1}{y^2}\right)((x+y)^2-(x-y)^2)((x+y)^2+a)((x-y)^2+a)\\
&-\frac{2}{xy}\left[(x-y)^2((x+y)^2+a)^2+(x+y)^2((x-y)^2+a)^2\right]\\
\end{align*}
It is easy to see this is a quadratic polynomial in $a$. Let's call the coefficient of the second order term $A_2$ , then
\begin{align*}
A_2&=\left(\frac{1}{x^2}+\frac{1}{y^2}\right)((x+y)^2-(x-y)^2)-\frac{2}{xy}[(x-y)^2+(x+y)^2]\\
&=\left(\frac{1}{x^2}+\frac{1}{y^2}\right)\cdot 4xy-\frac{2}{xy}[2x^2+2y^2]\\
&=\frac{4}{x^2y^2}((x^2+y^2)xy-xy(x^2+y^2))\\
&=0.
\end{align*}
Similarly, for coefficient of the first order term $A_1$, we have
\begin{align*}
A_1&=\left(\frac{1}{x^2}+\frac{1}{y^2}\right)(4xy)((x+y)^2+(x-y)^2)-\frac{2}{xy}[2(x-y)^2(x+y)^2+2(x+y)^2(x-y)^2]\\
&=\frac{1}{x^2y^2}[(x^2+y^2)^2\cdot 8xy-8xy(x^2-y^2)^2]\\
&\geq 0.
\end{align*}
Lastly, for the coefficient of the constant term $A_0$, we have
\begin{align*}
A_0&=\left(\frac{1}{x^2}+\frac{1}{y^2}\right)(4xy)(x+y)^2(x-y)^2-\frac{2}{xy}[(x-y)^2(x+y)^4+(x+y)^2(x-y)^4]\\
&=\frac{(x+y)^2(x-y)^2}{x^2y^2}[(x^2+y^2)\cdot 4xy-2xy((x+y)^2+(x-y)^2)]\\
&=0.\\
\end{align*}
In all, we have $\partial_a G(x,y,a)\geq 0$ for $x,y>0$.
\end{proof}

From this lemma, one can do the same argument to get the blow up result, which is the following theorem:
 \begin{theorem}
\label{periodic1}
There exists initial data such that solutions to $\eqref{model1}$ and $\eqref{model2}$, with velocity given by $\eqref{eq:general}$, and $F(x,y,a)$ defined by (\ref{Fnon}), blow up in finite time.
\end{theorem}
In fact, we can prove the following type of initial data will lead to blow up:
\begin{itemize}
\item $\theta_{0x}, \omega_0$ smooth odd and are supported in $[-1,1]$.
\item $\theta_{0x}, \omega_0 \ge 0$ on $[0,1]$.
\item $\theta_0(0)=0$.
\item $\|\theta_0\|_\infty \le M$.
\end{itemize}

And similarly, for general pertubation (analogue of theorem \ref{genper}), we also have the similar blow up result.

Assume the velocity $u$ is given by the following choice of Biot-Savart Law
\begin{align}\label{gennon1}
u(x) &= \frac{1}{\pi} \int_{-1}^1 \left( \log|(x-y)]| + f(x,y)\right) \omega(y)\, dy,\\
\end{align}
where $f$ is a smooth function whose precise properties we will specify later. We view $f$ as a perturbation and we will  show solutions to the system (\ref{model1}) and (\ref{model2})
can still blow up in finite time.

\begin{theorem}\label{gennon}
Let $f\in C^2$ be supported on $[-1,1]$, such that $f(x,y)=f(-x,-y)$ for all $y$. Then there exists initial data $\omega_0$, $\theta_0$  such that solutions of $\eqref{model1}$ and $\eqref{model2}$, with velocity given by (\ref{gennon1}), blow up in finite time.
\end{theorem}

\noindent
Again we can prove the following type of initial data will form finite time singularity:

\begin{itemize}
\item $\theta_{0x}, \omega_0$ smooth odd and are supported in $[-1,1]$.
\item $\theta_{0x}, \omega_0 \ge 0$ on $[0,1]$.
\item $\theta_0(0)=0$.
\item $\supp \omega_0 \subset [0,\epsilon]$.
\item $\|\theta_0\|_\infty \le M$.
\end{itemize}

We leave the proofs of these theorems as exercises for interested reader.

\textbf{Acknowledgment.}
TD and AK acknowledge partial support of the NSF-DMS grant 1412023. XX acknowledges partial support of the NSF-DMS grant 1535653.


\end{document}